\documentclass[11pt, letterpaper]{article}

\usepackage[english]{babel}	
\usepackage[utf8]{inputenc}
\usepackage{fontenc}
\usepackage{amsmath}
\usepackage{amssymb}
\usepackage{amsthm}	
\usepackage{amsfonts}
\usepackage{newpxtext}
\usepackage{newpxmath}
\usepackage[cal=euler]{mathalpha}
\usepackage[top=1.25in, bottom=1.25in, left=1.25in, right=1.25in]{geometry}
\usepackage[shortlabels]{enumitem}
	\setlist[itemize,1]{label=$\triangleright$}
	\setlist[itemize,2]{label=$\triangleright\triangleright$}
\usepackage{xcolor}
	\definecolor{myosotis}{RGB}{92, 84, 148}
\usepackage{cite}
\usepackage{calc}
\usepackage{colortbl}
\usepackage{cases}
\usepackage{hyperref}	
	\hypersetup{
		colorlinks=true,
		linkcolor=myosotis,
		citecolor=myosotis,
	}
\usepackage[noabbrev,capitalize,nameinlink]{cleveref}

\newcommand{\N}{\mathbb{N}}

\newcommand{\inv}{^{\raisebox{.2ex}{$\scriptstyle-1$}}}
\newcommand{\set}[1]{\left\{#1\right\}}
\newcommand{\rest}[1]{|_{#1}}
\newcommand{\cl}[1]{\overline{#1}}
\newcommand{\subs}{\subseteq}

\newcommand{\subsn}{\subsetneq}
\newcommand{\I}{\mathcal{I}}
\newcommand{\mc}[1]{\mathcal{#1}}
\newcommand{\m}{\mathfrak{m}}
\newcommand{\Ldiv}{\L_{\! \operatorname{div}}}
\newcommand{\eq}{^{\operatorname{eq}}}

\newcommand{\clalg}[1]{\cl{#1}^{\text{a}}}

\newcommand{\PACV}{\mathrm{PACV}\!}
\newcommand{\K}{\mathbf{K}}
\newcommand{\LK}{\L_{\K}}
\newcommand{\LKi}{\L_{\K,i}}
\newcommand{\LG}{\L^{\hspace{-2pt}\mathcal{G}}}
\newcommand{\T}{\mathbf{T}}

\renewcommand{\L}{\mathcal{L}}
\renewcommand{\phi}{\varphi}
\renewcommand{\epsilon}{\varepsilon}
\renewcommand{\tilde}{\widetilde}
\renewcommand{\O}{\mathcal{O}}
\renewcommand{\v}{\mathrm{v}}
\renewcommand{\S}{\mathbf{S}}

\def\Ind#1#2{#1\setbox0=\hbox{$#1x$}\kern\wd0\hbox to 0pt{\hss$#1\mid$\hss}
	\lower.9\ht0\hbox to 0pt{\hss$#1\smile$\hss}\kern\wd0}
\def\ind{\mathop{\mathpalette\Ind{}}}

\DeclareMathOperator{\Lring}{\L_{\!\text{rg}}}
\DeclareMathOperator{\ACF}{ACF}
\DeclareMathOperator{\ACVF}{ACVF}
\DeclareMathOperator{\Th}{Th}
\DeclareMathOperator{\acl}{acl}
\DeclareMathOperator{\dcl}{dcl}
\DeclareMathOperator{\trdeg}{trdeg}
\DeclareMathOperator{\tp}{tp}
\DeclareMathOperator{\qftp}{qftp}
\DeclareMathOperator{\Aut}{Aut}
\DeclareMathOperator{\qf}{qf}
\DeclareMathOperator{\GL}{GL}
\DeclareMathOperator{\im}{im}
\DeclareMathOperator{\rad}{rad}

\theoremstyle{definition}
\newtheorem{df}{Definition}[section]
\newtheorem{remark}[df]{Remark}
\newtheorem{notation}[df]{Notation}

\theoremstyle{plain}
\newtheorem{teo}[df]{Theorem}
\newtheorem{coro}[df]{Corollary}
\newtheorem{lemma}[df]{Lemma}
\newtheorem{prop}[df]{Proposition}
\newtheorem{fact}[df]{Fact}

\newtheorem*{teoA}{Theorem A}
\newtheorem*{teoB}{Theorem B}

\author{Bryan González Leandro}
\title{Imaginaries in perfect bounded pseudo algebraically closed fields with finitely many independent valuations}
\date{}

\begin{document}
	
\maketitle

\begin{abstract}
	In this paper, we prove weak elimination of imaginaries for perfect bounded pseudo-algebraically closed fields equipped with finitely many independent valuations. Our approach combines an extension result for types to invariant types with an amalgamation theorem. As a special case, we obtain full elimination of imaginaries when the field is equipped with a single valuation.
\end{abstract}

\section*{Introduction}

Elimination of imaginaries (EI) is a fundamental structural property in model theory, concerned with the definable coding of equivalence classes within a structure. More explicitly, an $\L$-theory $T$ eliminates imaginaries if, for every model $M$ of $T$ and every $\L(A)$-definable relation $R$ on $M^n$, there exists an $\L(A)$-definable function $f$ such that $xRy$ if and only if $f(x)=f(y)$. This is equivalent to the existence, for every definable set $X$, of a \textit{canonical parameter}, that is, a tuple $c$ and a formula $\phi$ such that $\phi(M,d)=X$ if and only if $c=d$.

Given an $\L$-theory $T$, it is always possible to expand the language and the theory in order to obtain elimination of imaginaries without adding new structure that was not already interpretable. For this, we add a  new sort $S_R$ for each definable equivalence relation $R$ on a definable subset $X_R$ of $M^n$, a constant for each equivalence class of $R$, and a function symbol $f_R:X_R\to S_R$. The theory $T\eq$ is defined as the theory $T$, together with the axioms stating that that each function $f_R$ is onto, and $\forall xy(f_R(x)=f_R(y)\leftrightarrow xRy)$

Certainly, any model of $T$ can be expanded to a model of $T\eq$. Moreover, $T\eq$ eliminates imaginaries. In this way, model theorists are naturally led to describe which fragments of $\L\eq$  eliminate imaginaries. 

In the context of valued fields, significant progress was initiated by Haskell, Hrushovski, and Macpherson, who described in \cite{haskellhrushovskimacpherson} a language in which the theory of non trivially algebraically closed valued fields ($\ACVF$) eliminate imaginaries. We recall that the language $\Ldiv$ for a valued field $(F,\v)$ is the language of rings $\Lring$, together with a binary relation symbol $|$, such that  $x|y$ if and only if $\v(x)\leq \v(y)$. It was shown by A. Robinson in 1956 that the theory $\ACVF$ eliminates quantifiers in this language.

The \textit{geometric language} is a fragment of $\Ldiv\eq$ introduced in \cite{haskellhrushovskimacpherson} obtained by adding to $\Ldiv$ the codes of certain definable submodules of $F^m$, and codes for some equivalence classes inside these submodules. To make this precise, we add to $\Ldiv$ a sort $\S_m$ and a sort $\T_m$ for each $m\in \N^*$. The sort $\S_m$ is interpreted as $\S_m = \GL_m(F)/\GL_m(\O)$, where $\O$ is the valuation ring. Any element $s\in \S_m$ can be seen as a basis of the same $\O$-submodule of $F^m$, denoted by $\Lambda(s)$. Then, we can construct the quotient $\Lambda(s)/\m\Lambda(s)$, where $\m$ is the maximal ideal of $\O$. The sort $\T_m$ is interpreted as the set of pairs $(s, u + \m\Lambda(s))$, where $s\in \S_m$ and $u+\m\Lambda(s)$ is an element of $\Lambda(s)/\m\Lambda(s)$. In this language, $\ACVF$ has elimination of imaginaries. 

Building on this foundation, Rideau-Kikuchi and Montenegro \cite{ppc} showed elimination of imaginaries for a theory of bounded pseudo $p$-adically closed multi-valued fields with finitely many independent valuations. Their approach was to add a copy of the geometric language for each valuation, and to show that there are no new imaginaries that can be defined from the interactions between different valuations. Then, they used an abstract criterion for weak elimination of imaginaries, based on proving a result on extending certain types by invariant types, and then a 3-amalgamation theorem. 

In a complementary development, the same authors introduced the notion of pseudo $T$-closed structures in \cite{ptc}, formalizing the notion of being \textit{pseudo closed} with respect to a given theory of fields, and generalizing the definitions of pseudo algebraically closed, pseudo $p$-adically closed, and pseudo real closed fields. In particular, many of the key obstacles to extending the methods developed in \cite{ppc} to the case of pseudo algebraically closed fields were resolved.  

The goal of this document is to use the techniques in \cite{ppc} in order to prove weak elimination of imaginaries for a theory of bounded, perfect, pseudo algebraically closed fields with finitely many independent valuations ($\PACV_n$ for short, where $n$ is the number of valuations). 

\vspace{\baselineskip}

\textbf{Overview of the paper}

\vspace{\baselineskip}

We start by setting the preliminary concepts and results on valued fields and PAC fields in \cref{seccprel}, and in \cref{seccdeflenguaje} we define the language and the theory for which we will prove elimination of imaginaries. For a perfect bounded $\PACV_n$ field $(F,\v_1,\dots ,\v_n)$ we consider $\L_i$ as a copy of the geometric language for the valuation $\v_i$, together with a countable set of constants. We work with the language $\L=\bigcup_i\L_i$ and the theory $T$ obtained by adding to $\Th_{\L}(F)$ some axioms encoding the fact that the field is bounded using the constants we added to the $\L_i$. 

After setting our framework, we make a brief detour in \cref{secccodingfinite} to prove that, if we consider only one valuation in our language, finite sets are coded in the structure. This is a consequence of \cref{lemaintedef}, stating that elements in the sorts of a $\PACV_n$ field are interdefinable with some elements in the sorts of its algebraic closure, and concluding with elimination of imaginaries for $\ACVF$. 

In \cref{secccrit} we state the criterion for weak elimination of imaginaries that we will use, which is proved in \cite[Proposition 1.17 and Lemma 1.19]{ppc}. This criterion is based on the following notion of quantifier free invariant independence: for a tuple $a$ and sets $B,E$ we say that the relation $a\ind_{E}^{\mathrm{i},\qf}B$ holds if there exists an $E$-invariant type $p$ over $M$ such that $a\models p\rest{BE}$.

The rest of the paper is focused on proving the conditions in \cref{criterioeliminacion}, the aforementioned criterion. In order to do that, we first need to understand in what way the geometric sorts for different valuations interact. The main result in \cref{seccort} is \cref{ortoti}, stating that they do not interact with each other: they are orthogonal. If $t_1,\dots ,t_n$ are elements of the geometric sorts of different valuations,  then the type of $(t_1,\dots ,t_n)$ over $A$ is completely determined by the $\L_i$-types of each $t_i$ over $A$, provided that that $A$ is a substructure contained in the field sort, and containing all of the field elements of $\acl(A)$. 

After this, \cref{seccalgclos} is devoted to describing the algebraic closure in models of $T$ of a substructure $A$. On one hand, \cref{partecorpiqueacl} says that the field elements of $\acl(A)$ are algebraic in the field-theoretic sense. On the other hand, \cref{parteimaginariaacl} states that the geometric elements in $\acl(A)$ are $\L_i$-algebraic over the elements of $A$ in the sorts for the valuation $v_i$. In this way, we decompose $\acl(A)$ as a union of algebraic closures in the languages $\L_i$. 

\cref{seccdensidtipos} contains the first big step in order to prove weak elimination of imaginaries. For a structure $M$, we denote $\mc{I}(M/E)$ for the set of types in $\mc{S}(M)$ that are $E$-invariant. We prove the following theorem. 
\begin{teoA}
	\textup{(\cref{teodensidadtipos})}. 
	Let $A=\acl^{\eq}(A)\subs M^{\eq}$ and $c\in \K^m(M)$. Then, there are types $p_i\in\mc{I}(\clalg{M}/\mc{G}_i(A))$ such that the partial type $\tp(c/A)\cup\bigcup_i p_i$ is consistent.
\end{teoA}
This says that, any type over an $\acl\eq$-closed set of parameters is consistent with a union of quantifier free types, who are global with respect to only one valuation and invariant over the given set of parameters. This will imply the first condition in \cref{criterioeliminacion}.

The last step is to show a 3-amalgamation result, as it is described in the second condition of \cref{criterioeliminacion}. For this, we prove in \cref{qftpimpliesfulltype} that, if $a$ is a tuple and $E$ is a set of parameters containing the field part of $\acl(E)$, then the $\L$-type of $a$ over $E$ is determined by the quantifier free type over $E$ of $a$, together with the field elements of $\acl(Ea)$. The proof of \cref{3amalgamacion}, the main theorem of \cref{secc3amalg}, is based on proving a collection of linear independence statements for certain subfields, in order to amalgamate the induced structure on them using the results in \cite[Section 4]{ptc}, and concluding with \cref{qftpimpliesfulltype}.

Finally, in \cref{seccfianl} we state the main theorem of this text:
\begin{teoB}
	\textup{(\cref{teofinal})}.
	The theory $T$ weakly eliminates imaginaries.
\end{teoB}
We also conclude full elimination of imaginaries in the case of only one valuation.

\section*{Acknowledgments}

The author would like to express sincere gratitude to Silvain Rideau-Kikuchi for their guidance, support, and insightful discussions throughout this work. The author also acknowledges financial support from the Universidad de Costa Rica (UCR) through a doctoral scholarship.

\section{Preliminaries}\label{seccprel}

We assume that the reader is familiar with the basics of model theory and valued fields. If not, we refer the reader to \cite{tent} for an introduction to model theory, and \cite{valuedfieldsenglerprestel} for an introduction to the basic facts on valuations on fields. 

We recall that a variety $V$ over the field $F$ is \textit{geometrically integral} if its field of functions $F(V)$ is a regular extension of $F$.

\begin{fact}
	Let $F$ be a field. The following are equivalent:
	\begin{enumerate}
		\item For every geometrically integral variety $V$ defined over $F$, $V$ admits an $F$-rational point.
		\item For every geometrically integral variety $V$ defined over $F$, $V(F)$ is Zariski-dense in $V$.
		\item $F$ is existentially closed, as a structure in the language of rings, in any regular extension.
	\end{enumerate}
	We say that $F$ is \textbf{pseudo algebraically closed} (PAC) if it satisfies the conditions above. 
\end{fact}

One of the important results on valued PAC fields is the fact that, if we prescribe a valuation, the field is dense in its algebraic closure

\begin{fact}
	\cite[Proposition 11.5.3]{friedjarden}. If $F$ is a PAC field and $\v$ is a valuation on $\clalg{F}$, then $F$ is $\v$-dense in $\clalg{F}$.
\end{fact}

In this document, we will focus on PAC fields and we will consider finitely many independent valuations on them. As a reminder for the reader, two valuations are independent if they generate different topologies. As a consequence of the independence of the valuation, we have the following simultaneous approximation theorem.

\begin{fact}\label{approximationprestelziegler}
	\cite[Theorem 4.1]{prestelziegler}
	Let $(F,\v_1,\dots ,\v_n)$ be a $\PACV_n$ field. If, for $i=1,\dots ,n$, we consider $U_i$ an open set of the topology generated by $\v_i$, we have that $\cap U_i\neq\emptyset$. 
\end{fact}

Let us set some notation. Let $(F,\v_1,\dots ,\v_n)$ be a field with $n$ valuations on it. We write $\O_i$ for the valuation ring of $\v_i$, $\m_i$ for its maximal ideal, $\Gamma_i$ for its value group, and $\mathbf{k}_i$ for its residue field. The balls coming from the valuation $\v_i$ will be called $i$-balls. 

For the following, we consider only one valuation $\v$ on the field $F$. 

\begin{remark}
	Let $s\in \GL_m(F)/\GL_m(\O)$. One easily shows that the columns of any representative of $s$ generate a free $\O$-module of rank $m$, that will be denoted by $\Lambda(s)$. This is well defined: $s$ is exactly the set of matrices whose columns form a basis for $\Lambda(s)$.
\end{remark}

\begin{df}
	The \textbf{geometric language} $\LG$ for the valued field $(F,\v)$  consists of the following:
	\begin{itemize}
		\item A sort $\K$ equipped with the language $\Ldiv$
		\item Sorts $\S_m$ for $m\in\N^*$ to be interpreted as $\S_m(F) = \GL_m(F)/\GL_m(\O)$, and maps $s_m: \K^{m^2}\to \S_m$ for the natural projections.
		\item Sorts $\T_m$ for $m\in\N^*$ to be interpreted as 
		\[ \T_m(F) = \set{ (s, t) : s\in\S_m(F), t\in \Lambda(s)/\m\Lambda(s) }. \]
		We will also consider maps $t_m: \K^{m^2}\to \T_m$ taking a matrix $B\in\GL_m(F)$ with first column $u$, and sending it to the tuple $(\,s_m(B), u + \m\Lambda(s_m(B))\,)$.
		\item The necessary predicates to have quantifier elimination in the $\LG$-theory of algebraically closed valued fields ($\ACVF^{\mathcal{G}}$). 
	\end{itemize}
\end{df}

In this language, \cite{haskellhrushovskimacpherson} showed that the theory $\ACVF$ admits elimination of imaginaries. 

\begin{remark}\label{remarknoocupoS}
	Note that we can consider $\S_m$ as a subset of $\T_m$, by the identification $s\mapsto (s,\m\Lambda(s))$. This is a $\emptyset$-definable injection, therefore we will suppose that formulas in this language have parameters only in the $\T$ sorts and/or in the field sort. 
\end{remark}

\begin{remark}\label{remarkmetertodoentm}
	Take a formula $\phi(x_1,\dots ,x_{\ell})$, where each $x_j$ is a variable for the sort $\T_{m_j}$. Let 
	\[t_j= (s_j, u_j + \m\Lambda(s_j))\in \T_{m_j}.\]  
	Considering the module $\Lambda=\prod_j \Lambda(s_j)$, we can easily construct $s\in\S_m$ for $m=\sum_j m_j$ such that $\Lambda(s)=\Lambda$, and $\emptyset$-definable functions $\rho_j:\S_m\to\S_{m_j}$ such that $\rho_j(s)=s_j$. We can also construct $\emptyset$-definable surjections $\eta_j:\Lambda(s)/m\Lambda(s)\to \Lambda(s_j)/\m\Lambda(s_j) $ given by the corresponding coordinate projections. Defining $t=(s, u + \m\Lambda(s))$, where $u=(u_1,\dots , u_{\ell})$ and $g_j(t)=(\rho_j(s),\eta_j(u+m\Lambda(s)))$, we have that  $\models \phi(t_1,\dots ,t_{\ell})$ if and only if $\models \phi(g_1(t),\dots ,g_{\ell}(t))$. This says that the $\LG$-formulas with parameters can be supposed to have parameters in the field sort and only one parameter in some $T_m$, for $m$ sufficiently large.
\end{remark}

\section{The language and the theory}\label{seccdeflenguaje}

In this section we will define the language and the theory we will work with. For $(F,\v_1,\dots ,\v_n)$ a $\PACV_n$ field, we're going to denote $\LKi$ a copy of $\Ldiv$ for the valuation $\v_i$ and $\LK=\bigcup_i\LKi$. Recall that the definition of the class $\PACV_n$ requires the valuations to be independent.

\begin{notation}
	Fix $(F,\v_1,\dots ,\v_n)$ a bounded, perfect, $\PACV_n$ field. Also, we consider a fixed function $\mathfrak{d}:\N^*\to \N^*$. We define $\L_i$ as a copy of the geometric language for the valuation $\v_i$, together with constants $c_j$ in the field sort, for each $j>0$, such that $|c_j|=\mathfrak{d}(j)$. All the $\L_i$ share the field sort and the constants $c_j$. We write $\mathcal{G}_i$ for the sorts of $\L_i$, and $\mc{G}_i^{\im}$ for all sorts of $\L_i$ except $\K$. We define $T_{\mathrm{bd}}$ as the $\L_i$-theory containing the field axioms, and the following sentences for each $m>0$:
	\begin{itemize}
		\item the polynomial $r_m(x)= x^{\mathfrak{d}(m)} +\displaystyle\sum\limits_{j<\mathfrak{d}(m)}c_{m,j}x^j$ is irreducible,
		\item every separable polynomial of degree $m$ is split modulo $r_m$.
	\end{itemize}
	We define $\L=\cup_i\L_i$ and $T=T_{\mathrm{bd}}\cup\Th_{\L}(F)$. We denote $\mc{G}$ for the sorts of $\L$.
\end{notation}

\begin{remark}
	As it is stated in \cite[Remark 6.3]{ptc}, the theory $T_{\mathrm{bd}}$ conveys what it means to be bounded. Precisely, any model of $T_{\mathrm{bd}}$ is bounded, and any bounded model of $\Th_{\L}(F)$ can be made into a model of $T_{\mathrm{bd}}$ for some well chosen function $\mathfrak{d}$ and constants $c_j$. 
\end{remark}

Our goal is to show weak elimination of imaginaries for the theory $T$. We fix $M\models T$ sufficiently saturated and homogeneous. Note that   $\clalg{\K(M)}$ can be naturally made into an $\LK$-structure extending $\K(M)$: it suffices to extend each valuation $\v_i$ (whose extension will be also denoted by $\v_i$). Then, we can define the geometric sorts for $\clalg{\K(M)}$ and make it into an $\L$-structure that we will denote by $\clalg{M}$.  This structure will also be fixed throughout this document.

\begin{remark}
	If $F_0\leq M$ is a subfield of $\K(M)$ and $F_0\models T_{\mathrm{bd}}$, then $\clalg{K(M)}=K(M)\clalg{F_0}$ by \cite[Lemma 6.4]{ptc}. Using the argument in \cite[Lemma 2.46]{ppc}, we show that for any $\L$-substructure $A\leq \K(M)$, if $\clalg{A}\cap\K(M)\subs A$, then $\clalg{A}=\clalg{F_0}A$.
\end{remark}

\section{Coding of finite sets from one valuation}\label{secccodingfinite}

In the case where we consider only one valuation, we can deduce the coding of finite sets from elimination of imaginaries in $\ACVF$. For this, we will show that certain elements of the algebraic closure of $M$ are interdefinable with elements of $M$. This is an adaptation to our case of \cite[Lemma 2.35]{ppc}. For this section only, we suppose that we consider only one valuation on $M$.

\begin{lemma}\label{lemaintedef}
	Let $\epsilon\in\dcl_{\L}^{\clalg{M}}(M)$. Then, there is $\eta\in M$ such that $\epsilon$ and $\eta$ are interdefinable in the pair $(\clalg{M},M)$.  

\end{lemma}

\begin{proof}
	Let $F_0$ be the subfield of $\K(M)$ generated by the constants $c_j$. Suppose $\epsilon\in\K(\clalg{M})$. Since we can write  $\clalg{\K(M)}$ as $\K(M)\clalg{F_0}$, there is $b\in\clalg{F_0}$ such that $\K(M)[\epsilon]=\K(M)[b]$. Note that $b$ is $\L_i$-definable without parameters, since the henselian closure of $F_0$ is $\clalg{F_0}$ \cite[Corollary 11.5.5]{friedjarden}. Then, we can find $\eta_i$ such that $\epsilon=\sum_i\eta_i b^i$. It suffices to consider $\eta$ as the tuple of the $\eta_i$. 
	
	By \cref{remarknoocupoS}, the only case left to consider is when $\epsilon\in\T_m(\clalg{M})$ for some $m$, write $\epsilon = (s, u + \m\Lambda(s))$. Let us define $L$ as a finite extension of $\K(M)$ where $s$ has a representative, and  $u+\m\Lambda(s)$ has a point. Without loss of generality, we suppose $u$ is that point. Let $a\in \clalg{F_0}$ such that $\O(L)=\O(\K(M))[a]$. Let $r$ be the degree of $a$ over $\K(M)$.
	
	Note that there is a vector space isomorphism $f_a:\K(M)^r\to L$ given by $x\mapsto \sum_{i=0}^{r-1} x_ia^i$.  Then, $f_a\inv(u+\m\Lambda(s))$ can be written as a coset $u' +  \m\Lambda(s')$, where $s'\in S_{mr}(M)$. Let $\eta$ be $(s', u'+\m\Lambda(s'))$. We will show that $\epsilon$ and $\eta$ are interdefinable. 
	
	Note that if $a$ and $a'$ are conjugate over $\Aut(\clalg{M}/F_0)$, then they are conjugate over $\Aut(\clalg{M}/M)$. This is, we can find $\sigma\in\Aut(\clalg{M}/M)$ such that $\sigma(a)=a'$. Then, a computation shows that
	\[ f_a\inv(u+\m\Lambda(s)) = f_{\sigma(a)}\inv(u+\m\Lambda(s)) = f_{a'}\inv(u+\m\Lambda(s)).\]
	This proves that the construction of $\eta$ remains the same if we take instead of $a$ some other conjugate of $a$ over $\Aut(\clalg{M}/M)$. Then, $\epsilon$ and $\eta$ are interdefinable in the pair $(\clalg{M}/M)$. Indeed, any automorphism $\tau$ of $\clalg{M}$ that fixes $M$ as a set fixes $\epsilon$ if and only if it fixes $\eta$. 
\end{proof}

\begin{coro}\label{codingfinitesets}
	Finite sets in $M$ are coded in $M$.
\end{coro}

\begin{proof}
	Any finite set in $\mc{G}(M)$ is coded, in $\clalg{M}$, by a tuple $\epsilon\in\dcl_{\L}^{\clalg{M}}(M)$. We conclude by \cref{lemaintedef}.
\end{proof}

\section{The criterion for weak elimination of imaginaries}\label{secccrit}

Let us now introduce the criterion for weak elimination of imaginaries that we are going to use. The criterion is based on the notion of quantifier free invariant independence. 

\begin{df}
	Let $A\subs M$. A global type $p(x)$ is said to be $\Aut(M/A)$-invariant if for every $\sigma\in\Aut(M/A)$ and $\phi(x,m)\in p(x)$, where $m$ is the tuple of parameters of the formula $\phi$, we have that $\phi(x,\sigma(m))$ is in $p(x)$. The set of $\Aut(M/A)$-invariant types will be denoted as $\I(M/A)$. 
\end{df}

\begin{df}
	Let $a\in M$ be a tuple and $B,E\subs M$. We define $a\ind^{\mathrm{i},\qf}_E B$ to hold if there is a quantifier free $\Aut(M/E)$-invariant type $p$ over $M$ such that $a\models p\rest{BE}$. In other words, $\qftp(a/BE)$ has an $\Aut(M/E)$-invariant, quantifier free, global extension.
\end{df}

\begin{prop} \label{criterioeliminacion}
	\cite[Proposition 1.17 and Lemma 1.19]{ppc}
	Consider a language $\L_0$, and let $\L_1$ be a fragment of $\L_0\eq$ containing $\L_0$ with sorts $R$. Let $T$ be an $\L_1$-theory and $M\models T$ sufficiently saturated and homogeneous. Write $M_0$ for $M\rest{\L_0}$. Suppose that:
	\begin{enumerate}
		\item for all $E=\acl\eq(E)\subs M\eq$, tuple $a\in M_0$ and $C\subs M_0$, there exists $a^*$ such that $a^*\equiv_{\L_1(E)}a$ and $a^*\ind_{R(E)}^{\mathrm{i},\qf}C$;
		\item for all $E=\acl_{\L_1}(E)\subs M$ and tuples $a_1,a_2,c_1,c_2,c\in M_0$, if $a_1\ind_E^{\mathrm{i},\qf} a_2$, $c\ind_E^{\mathrm{i},\qf} a_1a_2$,  $c_1\equiv_{\L_1(E)}c_2$ and for all $i$, $a_ic_i\equiv_{\L_1(E)}^{\qf}a_ic$, then there is $c^*$ such that $a_ic_i\equiv_{\L_1(E)}a_ic^*$ for all $i$.	
	\end{enumerate}
	Then, $T$ weakly eliminates imaginaries.
\end{prop}	

In our application to $\PACV_n$ fields, $\L_0$ will be $\LK$ and $\L_1$ will be $\L$. In this way, the tuples in the conditions of the criterion will be tuples in the field sort. 

\section{Orthogonality of the geometric sorts}\label{seccort}

Intuitively, since the valuations are independent, one valuation in the language does not carry information about another one. This translates to the geometric structure. The goal of this section is to prove that the geometric sorts for different valuations do not "interact" with each other (i.e. they are orthogonal). 

\begin{lemma}\label{simaxtrdeg}
	Fix $m\in\N^*$, $A\subs \K(M)$ and for $i=1,\dots ,n$ let $s_i\in \mathbf{S}_m^i(M)$, Assume $M$ is $|A|^+$-saturated. Then there exists $c\in \bigcap_i s_i$ such that $\trdeg(c/A)=m^2$. 
\end{lemma}

\begin{proof}
	The argument in this proof is the same as in \cite[Lemma 2.25]{ppc}, but we include the proof for the sake of completeness. First, note that $\O_i^{\times}$ is an $i$-open subset of $\K(M)$, since $\O_i^{\times}= \O_i - \m_i$. Then, $\GL_m(\O_i)\subs\K(M)^{m^2}$ is $i$-open too: it suffices to see it as the inverse image of $\O_i^{\times}$ under the determinant map. As each $s_i$ is just a traslate of $\GL_m(\O_i)$, they are also open subsets of $\K(M)^{m^2}$.
	
	Let $U_i$ be a product of closed $i$-balls such that $U_i\subs s_i$. Note that this exists since any non trivial open $i$-ball is also $i$-closed. Then, it suffices to find $c\in \bigcap_i U_i$ such that $\trdeg(c/A)=m^2$. By induction on the dimension and translation, this reduces to showing that, given any small set of parameters, we can find some element of $\bigcap_i\O_i$ that is transcendental over that set. This can be proved by compactness: by \cref{approximationprestelziegler}, $\bigcap_i\O_i$ is infinite. 
\end{proof}

\begin{lemma}\label{timaxtrdeg}
	Fix $m\in\N^*$, $A\subs \K(M)$ and for $i=1,\dots ,n$,  let $t_i=(s_i, u_i+ \m_i\Lambda(s_i))\in \mathbf{T}_m^i(M)$, Assume $M$ is $|A|^+$-saturated. Then there exists $c\in\bigcap_i s_i$ and $d\in\bigcap_i u_i+\m_i\Lambda(s_i)$  such that $\trdeg(cd/A)=m^2+m$. 
\end{lemma}

\begin{proof}
	The tuple $c$ is obtained applying \cref{simaxtrdeg}. Now, since $\m_i^m$ is an open subset of $\K(M)^m$, $m_i\Lambda(s_i)$ is also open, and therefore each $u_i+\m_i\Lambda(s)$ is too. We can apply the same argument used in \cref{simaxtrdeg}, \textit{mutatis mutandis}, to find $d\in\bigcap_i u_i+\m_i\Lambda(s_i)$ with $\trdeg(d/A(c))=m$. Then,
	\[ \trdeg(cd/A) = \trdeg(d/A(c)) + \trdeg(c/A) = m+m^2.\qedhere \]
\end{proof}

\begin{teo}\label{ortoti}
	Let $A=\K(\acl(A))\leq \K(M)$ and $t_i,t_i'\in\mathbf{T}_m^i(M)$ such that $t_i\equiv_{\L_i(A)}t_i'$ for all $i$.. Then, $(t_1,\dots ,t_n)\equiv_{\L(A)}(t_1',\dots ,t_n')$
\end{teo}

\begin{proof}
	Let
	\begin{align*}
		t_i &= (s_i , u_i + \m_i\Lambda(s_i)),\\
		t_i' &= (s'_i , u'_i + \m_i\Lambda(s'_i)),
	\end{align*}
	and $c,d$ as in \cref{timaxtrdeg} for the $t_i$. Since $t_i\equiv_{\L_i(A)}t_i'$, by homogeneity there is an $\L_i(A)$-automorphism $f_i$ of $M$ sending $t_i$ to $t_i'$. In particular, it sends $s_i$ to $s'_i$. We define $c_i = f_i(c)$ and $d_i=f_i(d)$.  Then, $(c_i,d_i)\equiv_{\L_i(A)}(c,d)$ and $(c_i,d_i)$ belongs to $s'_i\times (u'_i+\m_i\Lambda(s_i'))$. This is, $\tp_{\LKi}(c,d/A)$ is consistent with $s'_i\times (u'_i+\m_i\Lambda(s_i'))$ for each $i$. It follows by \cite[Proposition 6.11]{ptc} that $\tp_{\LK}(c,d/A)\cup\set{x\in\bigcap_i s'_i\times (u'_i+\m_i\Lambda(s_i'))}$ is consistent in $M$. If $(c^*,d^*)$ realizes this partial type, we have an automorphism of the field sort sending $(c,d)$ to $(c^*,d^*)$. Extending this map to the geometric structure , we get an $\L(A)$-automorphism $\sigma$ sending $c$ to $c^*$, so it sends $s_i=c\cdot \GL_{m}(\O_i)$ to $s_i'=c^*\cdot \GL_{m}(\O_i)$. Knowing this, we can conclude that this automorphism sends $d+\m_i\Lambda(s_i)$ to $d^*+\m_i\Lambda(s'_i)$, and therefore $\sigma(t_i)=t'_i$.
\end{proof}

\begin{coro}\label{coroideftorsors}
	Let $A\leq \mathbf{K}(M)$ and $X\subs \prod_i \mathbf{T}^i_{m_i}$ be an $\L(A)$-definable set. Then, there exists finitely many quantifier free $\L_i(A)$ formulas $\phi_{i,j}(x_i)$ such that $X=\bigcup_j \prod_i\phi_{i,j}(M)$. 
\end{coro}

\begin{proof}
	Let $x=(x_1,\dots ,x_n)$ and $x'=(x_1'\dots x_n')$, where $x_i,x_i'$ are variables for the sort $\mathbf{T}^i_{m_i}$. By \cref{ortoti}, the following set is inconsistent:
	\[ \set{\psi(x_i)\leftrightarrow \psi(x_i'):0\leq i\leq n; \psi\in\L_i(A)}\cup\set{(x_1,\dots ,x_n)\in X\land (x_1',\dots , x_n')\notin X}. \]
	This is inconsistent because an element of this set gives us two tuples $(t_1,\dots ,t_n)$ and $(t_1',\dots ,t_n')$ with $t_i,t_i'\in\T^i_{m_i}(M)$, such that, for each $i$, the elements $t_i$ and $t_i'$ have the same $\L_i(A)$-type, but the tuples don't have the same $\L(A)$-type. 
	
	By compactness, there is a finite subset that is inconsistent. That is, there are formulas $\psi_{i,j}(x_i)$ for $j\in\set{1,\dots , \ell_i}$ such that
	\[ M\models \forall x,x' \left( \bigwedge_{0\leq j\leq \ell_i}\psi_{i,j}(x_i)\leftrightarrow\psi_{i,j}(x_i')  \quad\rightarrow\quad x\in X \leftrightarrow x'\in X \right) \]
	
	We write $\psi^0$ for $\neg \psi$ and $\psi^1$ for $\psi$. For each $\epsilon_i:\ell_i\to 2$, let
	\[ \theta_{i,\epsilon_i}(x_i)=\bigwedge_{0\leq j\leq \ell_i}\psi_{i,j}(x_i)^{\epsilon_i(j)} \]
	and for all tuple $\epsilon=(\epsilon_1,\dots ,\epsilon_n)$, define $\theta_\epsilon(x)=\bigwedge_i \theta_{i,\epsilon_i}(x_i)$. We claim that  for all $\epsilon$, if $\theta_\epsilon(M)\cap X\neq\emptyset$, then $\theta_\epsilon(M)\subs X$. Suppose it is not the case. There is an $\epsilon$ and $x,x'$ such that $\models\theta_\epsilon(x)$, $\models\theta_\epsilon(x')$, the tuple $x$ belongs to $X$, but $x'\notin X$. By the definition of $\theta_\epsilon$, the formulas $\psi_{i,j}(x_i)^{\epsilon_i(j)}$ and $\psi_{i,j}(x_i')^{\epsilon_i(j)}$ are all satisfied, so $\psi_{i,j}(x_i)$ and $\psi_{i,j}(x_i')$ have the same truth values, which implies for this choice of $x$ and $x'$, that $\psi_{i,j}(x_i)\leftrightarrow \psi_{i,j}(x_i')$. But then $x\in X$ if and only if $x'\in X$, which is a clear contradiction. 
	
	Let $E=\set{\epsilon: \theta_\epsilon(M)\cap X\neq\emptyset}$.  We can write
	\[ X=\bigcup_{\epsilon\in E} \theta_\epsilon(M) = \bigcup_{\epsilon\in E}\prod_i \theta_{i,\epsilon_i}(M).\qedhere \]
\end{proof}

\begin{coro}\label{gordalobadealtoimpacto}
	Let $A\leq M$, $T_i$ be a product of sorts in $\mathcal{G}_i^{\im}$ and $X\subs \prod_i T_i$ an $\L(A)$-definable subset. Then, there exists quantifier free $\L_i(\mathcal{G}_i(A))$-definable sets $X_{i,j}\subs T_i$ such that $X=\cup_j\prod_i X_{i,j}(M)$. 
\end{coro}

\section{Description of the algebraic closure}\label{seccalgclos}

\begin{teo}\label{existclosedL*}
	Let $(L,\v_1,\dots ,\v_n)$ be a $\PACV_n$ field, and $L\subs L'$ a regular field extension. Then, $L$ is existentially closed in $L'$ as $\LK$-structures. 
\end{teo}

\begin{proof}
	This is a direct consequence of \cite[Theorem 5.11]{ptc}.
\end{proof}

\begin{prop}\cite[Proposition 6.9]{ptc}.
	Let $A\leq\K(M)$. Then, $\acl_{\LK}(A)\subs \clalg{A}$. 
\end{prop}

\begin{coro}\label{parteimaginariaacl}
	Let $A\leq M$. Then, for all $i$, we have $\mathbf{T}^i_m(\acl(A))\subs \acl_{\L_i}(\mathcal{G}_i(A))$.
\end{coro}

\begin{proof}
	Let $X\subs \T_m^i$ be a finite $\L(A)$-definable set. By \cref{gordalobadealtoimpacto}, $X$ is $\L_i(\mc{G}_i(A))$-definable. We can conclude that $X\subs \acl_{\L_i}(\mc{G}_i(A))$.
\end{proof}

\begin{prop}\label{partecorpiqueacl}
	Let $A\leq M$. Then, $\K(\acl(A))\subs\clalg{\K(A)}$.
\end{prop}

\begin{proof}
	Let $c\in \K(\acl(A))$. There exists $t_i\in \mathbf{T}^i_m$ for $m$ sufficiently large, such that $c\in \acl(\K(A)(t_1,\dots ,t_n))$. Write $t_i=(s_i, u_i+\m_i\Lambda(s_i))$. We apply \cref{timaxtrdeg} to find $d\in \bigcap_i s_i$ and $e\in\bigcap_i u_i+\m_i\Lambda(s_i)$ such that $\trdeg(de/\K(A)c)=m^2+m$. Clearly, $c\in\acl(\K(A)de))$. But then $c$ is a field element that is algebraic over field elements, so $c\in\acl_{\LK}(\K(A)de)\subs \clalg{\K(A)(de)}$. If we had that $c\in \clalg{\K(A)(de)}-\clalg{\K(A)}$, there would exist a polynomial $f(x)$ with at least one coefficient in $\K(A)(de)-\K(A)$, such that $f(c)=0$. This implies that $de$ is not transcendental over $\K(A)(c)$, a clear contradiction. We conclude that $c\in \clalg{\K(A)}$.
\end{proof}

\begin{coro}\label{funcimagenfinita}
	Let $R$ be a sort in $\mathcal{G}_i^{\im}$ and $R'$ a sort in $\mathcal{G}_j$ for $j\neq i$. Then any $\L(M)$-definable function $f:R\to R'$ has finite image.
\end{coro}

\begin{proof}
	If $R'=\T_m^i$, then this is a direct consequence of \cref{gordalobadealtoimpacto}. In the other case, we would have $f:R\to\K$. For any $a\in R(M)$, we have $f(a)\in\K(\dcl(Ma))\subs \K(M)$ by \cref{partecorpiqueacl}. Then, every elementary extension $N\geq M$ satisfies $f(R(N))\subs \K(M)$. If $f(R(M))$ was infinite, by compactness it would have a point in $N-M$, which is not possible. 
\end{proof}

\begin{coro}\label{coroaclLi}
	Let $A\leq M$. Then, $\acl(A)\subs \bigcup_i \acl_{\L_i}(\mathcal{G}_i(A))$. 
\end{coro}

\begin{proof}
	This folows immediatly from \cref{parteimaginariaacl} and  \cref{partecorpiqueacl}	
\end{proof}

\section{Extending $\L$-types by invariant $\L_i$-types}\label{seccdensidtipos}

The goal of this section is to prove that a type over an $\acl\eq$-closed set is consistent with some union over $i$ of invariant $\L_i$-types. This will imply the first condition in \cref{criterioeliminacion}. 

\begin{df}
	For $A\subs \clalg{M}$, we denote by $B_i(A)$ the set of all $A$-definable $i$-balls in $\clalg{M}$. If $P\subs B_i(A)$, we denote by $\eta_P$ the generic type of $\bigcap_{b\in P}b$. That is,
	\[ \eta_P(x) := \set{ x\in b:b\in P }\cup \set{x\notin b': b'\in B_i(\clalg{M}), b'\subsn \bigcap_{b\in P}b}.  \]
\end{df}

\begin{remark}
	The type $\eta_P$ defined above generates a complete type if it is consistent. Indeed, if $b_0$ is any $i$-ball, either the formula $x\in b_0$ belongs to $\eta_P$, or the formula $x\notin b_0$ belongs to $\eta_P$. By $C$-minimality of $\ACVF$, the type generated by $\eta_P$ is complete. For simplicity, we will not distinguish between $\eta_P$ and its generated type. 
\end{remark}

\begin{remark}
	Let $P\subs B_i(A)$, and $\sigma\in\Aut(\clalg{M}/A)$. Since all $i$-balls in $P$ are $A$-definable, the action of $\sigma$ on the set of formulas $\set{x\in b:b\in P}$ is the identity. On the other hand, for any ball $b'$ such that $b'\subs b$ for all $b\in P$, we have that $\sigma(b')\subs \sigma(b)=b$ for all $b\in P$. Clearly, $\sigma(b')$ is a definable ball. Combining these two facts, we confirm that $\eta_P\in \I(\clalg{M}/A)$.  
\end{remark}

\begin{lemma}\label{lemabolaminimalcerrada}
	Let $b_1,\dots ,b_m$ be disjoint $i$-balls. Then, there is a minimal closed $i$-ball $\tilde{b}$ that covers all of the $b_j$.
\end{lemma}

\begin{proof}
	If there is only one ball, we take the corresponding closed ball. Now suppose we have at least two of them. First, note that if $b$ and $b'$ are disjoint balls, for any $x\in b$ and $y\in b'$ the value of $\v_i(x-y)$ does not depend on the choice of $x$ or $y$. Indeed, take another point $x'$ in $b$, and suppose, without loss of generality, that  $\v_i(x-y)<\v_i(x'-y)$. Then,
	\[ \v_i(x-x') = \v_i(x-y+y-x') = \v_i(x-y)<\v_i(x'-y).  \]
	Write $\rad(b)$ for the radius of the ball $b$. This implies that $\rad(b)\leq \v_i(x-x')<\v_i(x'-y)$, contradicting the fact that $b$ and $b'$ are disjoint.
	
	As a result of this, we can define a notion of distance between two disjoint balls as $d(b,b'):= \v_i(x-y)$ where $x\in b$ and $y\in b'$. Now, consider 
	\[\gamma := \min\set{d(b_j,b_{j'}):j\neq j'}.\]
	Let $\ell,k$ be such that $\gamma=d(b_{\ell},b_k)$ and take $a\in b_{\ell}$. We claim that the ball $\tilde{b}=\cl{B}(a,\gamma)$ has the desired properties. 
	
	First we show that $\tilde{b}$ covers all the $b_j$. If $j\neq \ell$, for any $x\in b_j$ we have that  $\v_i(x-a)=d(b_j,b_{\ell})\geq \gamma$, so $x\in \tilde{b}$.  On the other hand, if $x\in b_{\ell}$, take any point $c\in b_k$. Since $b_k\cap b_{\ell}=\emptyset$, we have $\v_i(x-a)\geq \v_i(c-a)=\gamma$, so $x\in\tilde{b}$. Now we check that $\tilde{b}$ is in fact minimal. If $\delta\in\Gamma_i$ is such that $\delta>\gamma$, and we take $x\in b_k$, then $\gamma=\v_i(x-a)<\delta$, so $x\notin \cl{B}(a,\delta)$. This proves that $\tilde{b}$ is the minimal closed ball that covers all of the balls $b_j$.
\end{proof}

\begin{prop}\label{densidadtiposinvariantesunavariable}
	Let $A=\acl\eq(A)\subs M\eq$ and $c\in \K(M)$. For each $i\in\set{1,\dots ,n}$, let $P_i:=\set{b\in B_i(A):c\in b}$. Then, the partial type 
	\[ \tp(c/A)\cup \bigcup_i \eta_{P_i} \]
	is consistent.   
\end{prop}

\begin{proof}
	Suppose the type is not consistent. Then, there is an $\L(A)$-definable set $X$ with $c\in X$, balls $b_i\in P_i$ and $b_i^j\in B_i(\clalg{M})$ for $j=1,\dots , m_i$ with $b_i^j\subsn \bigcap_{b\in P_i}b$, such that
	\begin{equation}
		X\cap \bigcap_i b_i\subs \bigcup_{i,j}b_i^j.\label{arepaesta}
	\end{equation}
	
	We suppose, without loss of generality, that none of the balls $b_i^j$ is redundant. This means that for each ball $b_i^j$,
	\[ \left(  X\cap \bigcap_i b_i \right) - \bigcup_{k\neq i, \ell\neq j}b_k^{\ell}\neq \emptyset. \]
	As a remark, this implies that, for each $i$, the balls $b_i^j$ are disjoint. 
	
	Let us define a new set $X_{n-1}$ satisfying the conditions described above for $X$, but having no balls of the form $b_n^j$. If there is already no balls $b_n^j$ in the inclusion (\ref{arepaesta}), we take $X_{n-1}=X$. Otherwise, we consider the following two cases. For both of them, we fix the balls $b_i^j$ with $i<n$. 
	
	We suppose now that $P_n$ has a minimal ball that is closed. We can assume without loss of generality that $b_n$ is that ball.  Moreover, we can assume the balls $b_n^j$ to be open. To see this, define $\delta=\rad(b)$, and $a_j\in b_n^j$ any point. Then, we replace $b_n^j$ by the ball $B(a_j,\delta)$. In this process, it might be the case that for some $j$ and $j'$, the balls $B(a_j,\delta)$ and $B(a_{j'},\delta)$ are not disjoint. Then, they have to be equal, so it suffices to remove one of them from the inclusion. 
	
	Note that, in light of the construction above, we can assume that the balls $b_n^j$ are maximal for inclusion. As a final assumption, we will suppose that $m_n$ is the minimal amount of balls $b_n^j$ with the given characteristics: open, not redundant, pairwise disjoint, maximal for inclusion and satisfying (\ref{arepaesta})).
	
	Note that the set $\set{b_n^1,\dots b_n^{m_n}}$ is the only set of $n$-balls satisfying all of these assumptions and having size $m_n$. By \cref{codingfinitesets}, this says that we can construct an $\L(A)$-definable function $f$ from a product of sorts in $\mathcal{G}_i^{\im}$ where $i<n$, to a product of sorts in $\mathcal{G}_n$, such that
	\[f(\ulcorner b_i^j\urcorner:i<n, j\leq m_i)= \left\ulcorner\set{ b_n^j:j\leq m_n}\right\urcorner.\]
	By \cref{funcimagenfinita}, $f$ has finite image, so the codes of the balls $b_n^j$ are algebraic over $A$. Since $A=\acl(A)$, we get that the $b_n^j$ are $A$-definable. Note that this implies $c\notin \bigcup_j b_n^j$.  Otherwise, some $b_n^j$ would be in $P_n$, contradicting the minimality of $b_n$. In this case, we define
	\[ X_{n-1} := X - \bigcup_{j}b_n^j. \]
	
	The last case to consider is when $P_n$ does not have a minimal ball that is closed. We define $\tilde{b}_n$ as the  minimal closed ball that covers the balls $b_n^j$, given by \cref{lemabolaminimalcerrada}. Then, we have 
	\[ X\cap \bigcap_i b_i \subs\left( \bigcup_{i< n; j}b_i^j \right)\cup \tilde{b}_n. \]
	Note that $\tilde{b}_n\subsn \bigcap_{b\in P_n}b$. If the set
	\[ \left( X\cap\bigcap_i b_i \right) -  \bigcup_{i<n;j}b_i^j \]
	is just one point, let $b^*$ be that point. Otherwise, we define $b^*$ as follows. Consider the set $Z$ of closed $n$-balls $b'$ such that $b'\subs \bigcap_{b\in P_n}b$ and 
	\[  X\cap \bigcap_i b_i \subs\left( \bigcup_{i< n; j}b_i^j \right)\cup b'.\]
	Note that $Z\neq\emptyset$ since $\tilde{b}_n$ belongs to it.
	
	The set $\rad(Z)=\set{\rad(b):b\in Z}$ is bounded above by any value of the form $\v_n(x-y)$ with $x,y\in \left( X\cap\bigcap_i b_i \right) -  \bigcup_{i<n;j}b_i^j$. By o-minimality of DOAG, $\rad(Z)$ has a supremum $\gamma\in \Gamma_n$. Consider any point $a\in \left( X\cap\bigcap_i b_i \right) -  \bigcup_{i<n;j}b_i^j$ and take $b^*:=\cl{B}_n(a,\gamma)$. By definition of $\gamma$, the ball $b^*$ is the minimal ball of $Z$.
	
	As in the previous case, this gives us a way of constructing an $\L(A)$-definable function $f$ from a product of sorts in $\mathcal{G}_i^{\im}$ where $i<n$, to a product of sorts in $\mathcal{G}_n$, such that $f( b_i^j:i<n, j\leq m_i)= b^*$. By \cref{funcimagenfinita}, $f$ has finite image. So, $b^*$ is $A$-definable and then $c\notin b^*$: otherwise $b^*$ would be the minimal closed ball of $P_n$, a contradiction. In this case, we define
	\[ X_{n-1}:= X- b^*. \]
	
	Note that in all cases, $X_{n-1}$ is $\L(A)$-definable and $c\in X_{n-1}$. Moreover, 
	\[ X_{n-1}\cap  \bigcap_i b_i \subs \bigcup_{i\leq n-1; j}b_i^j, \]
	so we have the same conditions that we started with, only now the union on the right-hand side of the inclusion has no $n$-balls. By repeating this construction, we get to construct a set $X_0$ containing $c$, such that $X_0\cap\bigcap_i b_i\subs\emptyset$, which is a clear contradiction. 
\end{proof}

\begin{lemma}\label{pelosuelto}
	Let $A\subs M^{\eq}$ containing $\mc{G}(\acl^{\eq}(A))$, and $a\in\K(N)$ where $N\succeq M$. Suppose that $\tp_{\L_i}(a/\clalg{M})\in\I(\clalg{M}/\mc{G}_i(A))$. Then, $\tp_{\L_i}(\mc{G}_i(\acl^{\eq}(Aa))/\clalg{M})\in \I(\clalg{M}/\mc{G}_i(A))$
\end{lemma}

\begin{proof}
	By \cref{coroaclLi} and \cref{lemaintedef}, we can apply \cite[Corollary 1.10]{ppc} to $T_1 = \ACVF_i^{\mc{G}}$.
\end{proof}

\begin{teo}\label{teodensidadtipos}
	Let $A=\acl^{\eq}(A)\subs M^{\eq}$ and $c\in \K^m(M)$. Then, there are types $p_i\in\mc{I}(\clalg{M}/\mc{G}_i(A))$ such that the partial type $\tp(c/A)\cup\bigcup_i p_i$ is consistent. 
\end{teo}

\begin{proof}
	The proof goes by induction on $m$, the case $m=1$ is \cref{densidadtiposinvariantesunavariable}. Now, consider $a_0\in \K^m(M)$ and $c\in\K(M)$. By the induction hypothesis, there are types $p_i(x)\in\mc{I}(\clalg{M}/\mc{G}_i(A))$ and $a$ in some elementary extension of $M$ such that 
	\[ a\models \tp(a_0/A)\cup\bigcup_i p_i(x). \]
	We define $E_i=\mc{G}_i(\acl^{\eq}(Aa))$. By \cref{pelosuelto}, we have:
	\[ p'_i= \tp_{\L_i}(E_i/\clalg{M})\in \mc{I}(\clalg{M}/\mc{G}_i(A))\]
	
	Let $E=A\cup\bigcup_i E_i$. Then, $E=\acl^{\eq}(E)$. Now, consider $N\succeq M$ containing $E$. We apply \cref{densidadtiposinvariantesunavariable} to $c$ with $E$ as the set of parameters. We get types $p''_i\in\mc{I}(\clalg{N}/E_i)$ such that $\tp(c/E)\cup\bigcup_i p''_i$ is consistent (note that $\mc{G}_i(E)=E_i$). Let $c^*$ realize this partial type.  
	We define $q_i(x,y):=\set{ \phi(x,y):\phi(x,a)\in p''_i }$. Note that $(c^*,a)\models \tp(ca/A)\cup\bigcup_i q_i$. Indeed, $(c^*,a)\models q_i$ for all $i$ by definition of $q_i$ and $p_i''$, and $(c^*,a)\models \tp(ca/A)$ since $c^*\models \tp(c/E)$ and $E$ contains $a$.
	
	It remains to show that  $q_i\in\I(\clalg{M}/\mc{G}_i(A))$. Consider $\sigma\in\Aut_{\L_i}(\clalg{M}/\mc{G}_i(A))$. Recall that the type $\tp_{\L_i}(E_i/\clalg{M})$ is $\mc{G}_i(A)$-invariant. Then the partial map $\tau_0$ given by
	\[\tau_0(x)=\begin{cases}
		\sigma(x) & \text{ if }x\in\clalg{M},\\
		x & \text{ if } x\in E_i.
	\end{cases}\]
	is elementary. We consider $N^*\succeq \clalg{N}$ sufficiently strongly homogeneous, $\tau\in\Aut_{\L_i}(N^*/E_i)$ extending $\tau_0$, and $p^*\in\I(N^*/E_i)$ extending $p_i''$. Let $\phi(x,y,z)$ be an $\L_i$-formula and $u$ a tuple in $\clalg{M}$ such that $\phi(x,y,u)\in q_i$. By definition of $q_i$, we have that $\phi(x,a,u)\in p_i''$, and then this formula belongs to $p^*$ too. By $E_i$-invariance of $p^*$, the formula $\phi(x,\tau(a),\tau(u))\in p^*$. Since $a\in E_i$ and $u\in \clalg{M}$, this says that $\phi(x,a,\sigma(u))\in p_i''$. That is, $\phi(x,y, \sigma(u))\in q_i$.
\end{proof}

\section{The 3-amalgamation theorem over geometric parameters}\label{secc3amalg}

	In this section we will prove the second condition in \cref{criterioeliminacion} for the theory $T$. We will adapt the proof of \cite[Theorem 2.54]{ppc}, relying on some results proven in \cite{ptc} about amalgamation of $\LK$-structures. 

	\begin{df}
	 Let $L_1,L_2$ be fields with a common subfield $L_0$.
	 \begin{itemize}
		 	\item We say that $L_1$ is linearly disjoint from $L_2$ over $L_0$ if any subset of $L_1$ that is linearly independent over $L_0$, stays linearly independent over $L_2$ (seen inside $L_1L_2$). In this case, we write $L_1\ind_{L_0}^{\ell} L_2 $
		 	\item We say that $L_1$ is algebraically disjoint from $L_2$ over $L_0$ if any subset of $L_1$ that is algebraically  independent over $L_0$, stays algebraically independent over $L_2$ (seen inside $L_1L_2$). In this case, we write $L_1\ind_{L_0}^{a} L_2 $
		 	\item The extension $L_0\subs L_1$ is said to be regular if $L_1\ind_{L_0}^{\ell}\clalg{L_0}$. 
	 \end{itemize}
	\end{df}
	
\begin{lemma}\label{ppc2.47}
	Let $A\leq M$, $N\succeq M$, and $C\subs \K(N)$ containing $\K(A)$. Suppose that $\clalg{\K(A)}\cap M\subs A$ and that, for some $i$, $\qftp_{\L_i}(C/M)$ is $\Aut_{\L}(M/A)$-invariant. Then, $C$ and $\K(M)$ are algebraically disjoint over $\K(A)$.
\end{lemma}

\begin{proof}
	It is enough to show that for any finite tuple $u\in C$, its algebraic locus $V$ over $\K(M)$ is defined over $\K(A)$. Take an arbitrary $\sigma\in\Aut_{\L}(M/A)$. By invariance of $\qftp_{\L_i}(C/M)$, we have that $u\in\sigma(V)$, so $V\subs \sigma(V)$. We get the other inclusion using $\sigma\inv$ instead of $\sigma$. By elimination of imaginaries in $\ACF$, $V$ is defined over $\K(\dcl(A))$. Applying \cref{partecorpiqueacl}, we get $\K(\dcl(A))\subs \clalg{\K(A)}\cap M=\K(A)$. Then, $V$ is in fact defined over $\K(A)$. 
\end{proof}

\begin{lemma}\label{lemaisos1}
	Let $A,B \leq M$ be such that $\K(\acl(A))\subs A$ and $\K(\acl(B))\subs B$. Suppose that $f:A\to B$ is an $\L$-isomorphism. Then, $f$ is $\L$-elementary.
\end{lemma}

\begin{proof}
	Note that $f$ is, in particular, an $\LK$-isomorphism. By \cite[Proposition 6.7]{ptc}, $f\rest{\K(A)}$ is $\LK$-elementary. Then, $f\rest{\K(A)}$ extends to some $\sigma_0\in\Aut_{\LK}(\K(M))$. This induces an automorphism on the geometric structure, so we find $\sigma\in\Aut_{\L}(M)$ extending $f\rest{\K(A)}$.
	
	Let $a=(a_0,t_1,\dots ,t_n)$ be a tuple of elements in $A$, where $a_0$ is some tuple in $\K(A)$ and $t_i\in\T^i(A)$. We want to show that $(a_0,t_1,\dots ,t_n)\equiv_{\L}(f(a_0),f(t_1),\dots ,f(t_n))$. Note that $f(a_0)=\sigma(a_0)$. So, it suffices to show that $(f(t_1),\dots ,f(t_n))\equiv_{\L(\K(B))}(\sigma(t_1),\dots ,\sigma(t_n))$.
	
	For any tuple $b\in\K(B)$, we have $f\inv(b)=\sigma\inv(b)$. Seeing $f$ and $\sigma$ as $\L_i$-isomorphisms, we get that
	\[  (b,f(t_i))\equiv_{\L_i}^{\qf} (f\inv(b),t_i) = (\sigma\inv(b),t_i)\equiv_{\L_i}^{\qf} (b,\sigma(t_i)).  \]
	Then, $f(t_i)\equiv_{\L_i(\K(B))}\sigma(t_i)$. By \cref{ortoti}, we conclude that 
	\[(f(t_1),\dots ,f(t_n))\equiv_{\L(\K(B))}(\sigma(t_1),\dots ,\sigma(t_n)).\qedhere\]
\end{proof}

\begin{coro}\label{qftpimpliesfulltype}
	Let $E\subs M$ containing $\K(\acl(E))$ and $a\in M$ a tuple. Then
	\[ \qftp_{\L}(a\K(\acl(Ea))/E)\vdash \tp_{\L}(a/E) \]
\end{coro}

\begin{proof}
	If $b\K(\acl(Eb))\models\qftp_{\L}(a\K(\acl(Ea))/E) $, we have an $\L(E)$-isomorphism $f$ sending $a$ to $b$, and $\K(\acl(Ea))$ to $\K(\acl(Eb))$. By \cref{lemaisos1}, $f$ is elementary.  
\end{proof}

\begin{remark}\label{rmkaclclosedregular}
	Let $A\subs M$ such that $\K(\acl(A))=\K(A)$. This implies that $\K(A)$ is a perfect field, so the extension $\K(A)\subs\clalg{\K(A)}$ is separable. Clearly, any element of $ \clalg{\K(A)}\cap\K(M)$ is in $\K(\acl(A))=\K(A)$. Then, $\clalg{\K(A)}\cap\K(M)=\K(A)$. We conclude that the extension $\K(A)\subs \K(M)$ is regular. 
\end{remark}

\begin{teo}\label{3amalgamacion}
	Let $E\subs M$, and $a_1,a_2,c_1,c_2\in\K(M)$ be tuples such that:
	\begin{itemize}
		\item $\K(\acl(E))=\clalg{\K(E)}\cap M\subs E$,
		\item $a_j$ enumerates $A_j=\K(\acl(\K(E)A_j))=\clalg{\K(E)a_j}\cap M$,
		\item $c_j$ enumerates $C_j=\K(\acl(\K(E)C_j))=\clalg{\K(E)c_j}\cap M$,
		\item $A_1\cap A_2=\K(E)$,
		\item $c_1\equiv_{\L(E)}c_2$.	
	\end{itemize}
	Suppose there is a tuple $c$ in some elementary extension of $M$ such that:
	\begin{itemize}
		\item $c$ enumerates $C=\K(\acl(\K(E)C))=\clalg{\K(E)c}\cap M$,
		\item $c\ind_E^{\mathrm{i},\qf}M$,
		\item for all $j$, we have $a_jc_j\equiv_{\L(E)}^{\qf}a_jc$.
	\end{itemize}
	Then, the type
	\[ \tp(c_1/Ea_1)\cup\tp(c_2/Ea_2)\cup\qftp(c/M) \]
	is consistent. 
\end{teo}

\begin{proof}	
	By \cref{rmkaclclosedregular},  the fields $A_1$ and $A_2$ are regular extensions of $\K(E)$. Since $A_1\subs \K(M)$ and $A_2\subs\K(M)$ are regular, we deduce that $A_1A_2$ is a regular extension of $A_1$ and of $A_2$. Then, by \cite[Lemma 2.1]{wfree}, we have that 
	\[\clalg{A_1}\cap\clalg{A_2}=\clalg{A_1\cap A_2}=\clalg{\K(E)}.\]  
	
	Since $c\ind^{\mathrm{i},\qf}_E M$, we deduce that $C\ind_{\K(E)}^{a}\K(M)$ by \cref{ppc2.47} . We will show that $c$ realizes the desired type. Note that it suffices to show that $c\models \tp(c_1/Ea_1)\cup\tp(c_2/Ea_2)$.
	
	By hypothesis, $C_j\equiv_{\L(A_j)}^{\qf}C$, so there is a $\L(A_j)$-isomorphism $\phi_j:A_jC_j\to A_jC$ sending $C_j$ to $C$. This can be extended to an $\Lring(A_j)$-isomorphism $\clalg{\phi}_j:\clalg{A_jC_j}\to \clalg{A_jC}$. We define:
	\[ B_j:=\clalg{A_jC_j}\cap M,\qquad A:=\clalg{A_1A_2}\cap M,\qquad D_j:=\clalg{\phi}_j(B_j). \]
	We endow $D_j$ with the $\LK$-structure making $\clalg{\phi}_j$ an $\LK$-isomorphism. This is possible since $\LK$ is a relational expansion of $\Lring$.
	
	By \cite[Lemma 2.5.(2)]{wfree} , we get:
	\begin{align*}
	\clalg{CA_1}\cap \clalg{CA_2}\,\clalg{A_1A_2} &=\clalg{\clalg{C}\clalg{A_1}} \cap \clalg{\clalg{C}\clalg{A_2}}\, \clalg{\clalg{A_1}\clalg{A_2}}\\
	&= \clalg{\clalg{C}(\clalg{A_1}\cap\clalg{A_2})}\clalg{A_1}\\
	&= \clalg{C}\clalg{A_1}\\
	&= \clalg{\K(E)}CA_1. 
	\end{align*}
	Note that $D_2\subs \clalg{CA_2}$ and $A\subs \clalg{A_1A_2}$. Then, $D_2A\subs \clalg{CA_2}\clalg{A_1A_2}$, and this implies $D_2A= D_2A\cap \clalg{CA_2}\clalg{A_1A_2} $. Therefore, 
	\[ \clalg{CA_1}\cap D_2A = \clalg{CA_1}\cap D_2A\cap \clalg{CA_2}\,\clalg{A_1A_2}= D_2A\cap \clalg{\K(E)}CA_1 .\]
	
	On the other hand, since $\K(\acl(E))=\K(E)$, the extension $\K(E)\subs \K(M)$ is regular by \cref{rmkaclclosedregular}, and so are $\K(E)\subs D_2A$ and $\K(E)\subs CA_1$. We have then $D_2A\ind_{CA_1}^{\ell} \clalg{\K(E)}CA_1$. Then, $\clalg{CA_1}\cap D_2A = CA_1$. This is, $CA_1\subs D_2A$ is regular. Similarly one shows $CA_2\subs D_1F$ is regular. Note that, since $D_j\subs \clalg{CA_j}$, then $\ D_kA\ind_{CA_j}^{\ell}{D_j}$ for $j\neq k$. This implies that $D_1A\ind_{AC}^{\ell} D_2A$
	
	Now, we have that $C\ind_{\K(E)}^{\ell}\K(M)$. Certainly, it was demonstrated at the beginning that  $C\ind_{\K(E)}^{a}\K(M)$, and the extension $\K(E)\subs \K(M)$ is regular. It follows that $CA\ind_{A}^{\ell} \K(M)$, so $\clalg{CA}\ind_A^{a}\K(M)$. Since $AD_j\subs \clalg{CA}$, we have that $AD_j\ind_A^{a}\K(M)$. Using the fact that $A\subs \K(M)$ is a regular extension, we can deduce that $AD_j$ is linearly disjoint from $\K(M)$ over $A$. 
	
	With all this, we consider the following tensor product:
	\[ (\K(M)\otimes_{A}AD_1)\otimes_{\K(M)\otimes_{^A}AC}(\K(M)\otimes_{A}AD_2). \]
	This is isomorphic to the following tensor product:
	\[ (\K(M)\otimes_{A}AD_1)\otimes_{\K(M)\otimes_A AC}\left[ \K(M)\otimes_A (AC\otimes_{AC} AD_2) \right]. \]
	By  \cite[Exercise 2.15]{atiyah}, we have that the above tensor product is isomorphic to
	\begin{align*}
		\left[ (\K(M)\otimes_{A}AD_1)\otimes_{\K(M)\otimes_A AC} (\K(M)\otimes_A AC)\right]\otimes_{AC} AD_2
		&\simeq (\K(M)\otimes_A AD_1)\otimes_{AC} AD_2\\
		&\simeq \K(M)\otimes_A(AD_1\otimes_{AC}AD_2).
	\end{align*}
	Using the linear independences that we proved, we deduce that
	\[ \K(M)\otimes_A(AD_1\otimes_{AC}AD_2) \simeq \K(M)\otimes_A AD_1D_2 \simeq \K(M)D_1D_2 .\]
	Also, we get the following isomorphism:
	\[ (\K(M)\otimes_{A}AD_1)\otimes_{\K(M)\otimes_{A}AC}(\K(M)\otimes_{A}AD_2) \simeq \K(M)D_1\otimes_{\K(M)C}\K(M)D_2. \]
	Putting everything together, we have that
	\[ \K(M)D_1\otimes_{\K(M)C}\K(M)D_2 \simeq \K(M)D_1D_2. \]
	This is, $\K(M)D_1$ is linearly disjoint from $\K(M)D_2$ over $\K(M)C$. By \cite[Lemma 4.8]{ptc}, the field $\K(M)D_1D_2$ can be endowed with an $\LK$-structure extending that of $\K(M)D_1, \K(M)D_2$ and $\K(M)C$. 
	
	By \cref{existclosedL*}, $\K(M)$ is $\LK$-existentially closed in $\K(M)D_1D_2$. Then, we can find $L$ an $\LK$-elementary extension of $\K(M)$ containing $\K(M)D_1D_2$ as a substructure. Let $M^*$ be the $\L$-structure whose underlying field is $L$. Since the $\LK$-structure on $\K(M)D_1D_2$ extends that of $\K(M)C$, we have that  $\qftp_{\L}^{M^*}(c/M)=\qftp_{\L}^{M}(c/M)$. In particular, we still have $c\equiv_{\L(A_j)}^{\qf}c_j$.
	
	Note that $\clalg{D_j}=\clalg{\K(E)}D_j$ since $\clalg{B_j}=\clalg{\K(E)}B_j$. Now, $\K(E)\subs L$ is still regular, so $\clalg{D_j}\cap L= D_j$ by \cite[Fact 2.45.(i)]{ppc}). This is, $D_j=\K(\acl(D_j))=\K(\acl(CA_j))$. Also, $D_j\equiv_{\LK(A_j)}^{\qf}B_j$ by construction. Then, by \cref{qftpimpliesfulltype}, $c\equiv_{\L(A_j)}c_j$.
\end{proof}

\begin{remark}
	Montenegro and Rideau-Kikuchi proved in \cite[Theorem 2.54]{ppc} the analogous version of the above theorem for the case of pseudo $p$-adically closed fields. However, as it was pointed out by them, there is a mistake in the argument, since the structure constructed on $\K(M)D_1D_2$, in their proof, extends the structure on $\K(M)$ and on $C$ separately. This does not guarantee that $\qftp(c/Ea_1a_2)$ is preserved once we take the elementary extension $M^*$. However, their result is still valid: it suffices to use the tensor product argument given in the proof above.
\end{remark}

\section{Bringing everything together}\label{seccfianl}

\begin{teo}\label{teofinal}
	The theory $T$ weakly eliminates imaginaries.
\end{teo}	

\begin{proof}
	We apply \cref{criterioeliminacion}. Condition (2) is \cref{3amalgamacion}. To verify condition (1), we take $E=\acl^{\eq}(E)\subs M^{\eq}$, a tuple $a\in \K(M)$, and $C\subs M$. By \cref{teodensidadtipos}, we can find $\L_i$-types $p_i\in\I(\clalg{M}/\mc{G}_i(E))$ such that $\tp(a/E)\cup_i p_i$ is a consistent partial type. Let $a^*$ realize this type. Clearly, $a\equiv_{\L(E)} a^*$. Now, we also have that $\qftp_{\L}(a^*/\clalg{M})$ is $\Aut(\clalg{M}/E)$-invariant. Indeed, any atomic formula belonging to this type belongs to some $p_i$. This implies $a^*\ind_{\mc{G}(E)}^{\mathrm{i},\qf}C$. 
\end{proof}

\begin{coro}
	If the language contains only one valuation, then $T$ eliminates imaginaries. 
\end{coro}

\begin{proof}
	This is a direct consequence of \cref{teofinal} and \cref{codingfinitesets}.
\end{proof}


\nocite{*}
\bibliographystyle{amsalpha}
\bibliography{referencias.bib}

\end{document}